\documentclass[12pt,twoside]{amsart}
\usepackage{amssymb,amsmath,amsthm}
\usepackage{dsfont}
\usepackage{verbatim}
\usepackage{graphicx}
\usepackage{epsfig, enumerate}
\usepackage{color, soul}
\usepackage[all]{xy}

\def\Lip{{\rm Lip}}
\def\A{\mathcal A}
\def\B{\mathcal B}
\def\I{\mathcal I}

\def\L{\mathcal L}

\newtheorem{theorem}{Theorem}[section]

\newtheorem{lemma}[theorem]{Lemma}
\newtheorem{proposition}[theorem]{Proposition}

\newtheorem*{theo}{Theorem}

\newcounter{enumer}
\newtheorem{theoP}[enumer]{Theorem}
\newtheorem{propP}[enumer]{Proposition}
\theoremstyle{definition}
\newtheorem{definition}[theorem]{Definition}

\numberwithin{equation}{section}

\title{Lipschitz extensions of linear operators}

\author{Nahuel Albarrac\'{\i}n and Pablo Turco}

\address{IMAS - UBA - CONICET \\
Pab I,
Facultad de Cs. Exactas y Naturales \\Universidad de Buenos
Aires\\
Buenos Aires, Argentina}
\email{nalbarracin@dm.uba.ar}

\email{paturco@dm.uba.ar}


\begin{document}

\begin{abstract} 
Let $E, F, E_0$ be Banach spaces, with $E_0$ a subspace of $E$. For a maximal Banach operator ideal $\mathcal A$, we show that a linear operator from $E_0$ to $F$ can be extended to a linear operator from $E$ to $F$ that belongs to $\mathcal A$ if and only if it can be extended to a Lipschitz map from $E$ to $F$ belonging to a wide class of Lipschitz Banach operator ideals related with $\mathcal A$. As a consequence, we show that linear operators with special Lipschitz factorization through $\ell_{\infty}(\Gamma)$ has analogous linear factorization through $\ell_{\infty}(\Gamma)$.

\end{abstract}

\maketitle

\section*{Introduction}
The extension problem of a function between two Banach spaces (or, more generally, between any two spaces) has been studied by many authors over time. For a function $f\colon X\rightarrow Y$ and a superspace $Z\supset X$, there always exists a function $\widetilde f\colon Z\rightarrow Y$ such that $\widetilde f|_X=f$ without any additional requirements. Naturally, the problem becomes more relevant when $\widetilde f$ is expected to retain certain properties of $f$.

One of the cornerstone theorems in functional analysis is the Hahn--Banach Theorem. As a consequence of this, any continuous linear function from a Banach space to $\mathbb R$ admits an extension to any superspace, which is also linear and continuous with the same norm. In the same spirit, McShane \cite{McS} shows that every real-valued Lipschitz function on a metric space admits a Lipschitz extension to any superspace while preserving the Lipschitz constant. However, even for real-valued functions, certain types of functions may not admit a desired extension. For example, there is no Hahn--Banach-type theorem for $n$-homogeneous polynomials defined over Banach spaces (see \cite{KR}).

For vector-valued functions, it is well known that the identity map of $c_0$ cannot be extended to a linear (and continuous) function from $\ell_{\infty}$ to $c_0$. Furthermore, while the canonical inclusion $J_{c_0}\colon c_0\rightarrow \ell_{\infty}$ can be extended to a linear operator from $\ell_{\infty}$ to $\ell_{\infty}$ (since $\ell_{\infty}$ is an injective Banach space), no such extension can be factorized through $c_0$. Thus, a linear function (the inclusion of $c_0$ into $\ell_{\infty}$) with a desired property (factorization through $c_0$) may admit an extension that lacks this property.

For the linear case, Dom\'anski \cite{Dom} studied extensions of linear operators from a {\it Banach operator ideal} perspective. Also, Castillo, Garc\'{\i}a and Suarez \cite{CGS} study cases where linear operators belonging to the ideal of compact, weakly compact, or approximable operators can be extended to a linear operator which belongs to the same ideal.

There is a characterization for maximal Banach operator ideals whose operators admit extensions within the same ideal, expressed in terms of their associated tensor norm. As a consequence of the Representation Theorem for Maximal Operator Ideals (see ,e.g. \cite[Theorem~17.5]{DF}), for a maximal Banach operator ideal $\A$ there exists a {\it finitely generated} tensor norm $\alpha$ such that for all Banach spaces $E$ and $F$, $\A(E,F')=(E \widehat{\otimes}_{\alpha} F)'$, where $F'$ stands for the dual space of $F$. If $\alpha$ satisfies the property that $E_0 \widehat{\otimes}_{\alpha} F$ is a subspace of $E \widehat{\otimes}_{\alpha} F$ whenever $E_{0}$ is a subspace of $E$ (i.e., $\alpha$ is {\it left-injective}), then an application of the Hahn--Banach theorem ensures that every linear operator in $\A(E_0,F)$ has an extension in $\A(E,F)$. For instance, since the injective norm $\varepsilon$ is left-injective and $\mathfrak I(E,F')=(E\widehat{\otimes}_{\varepsilon} F)'$, where $\mathfrak I$ is the Banach ideal of integral operators, it follows that for any Banach spaces $E_0, E$ and $F$, $E_0$ being a subspace of $E$, if $T\colon E_0\rightarrow F'$ is an integral operator, there exists an extension of $T$ from $E$ to $F'$ that is also an integral operator with the same integral norm.

In the nonlinear case, a similar approach has been taken for $n$-homogeneous polynomials by Carando, Kirwan, Lassalle, Ryan, and Zalduendo \cite{CaLa, CaZa, KR, Za}, among others. Recently, Ch\'avez-Dom\'{\i}nguez and Jim\'enez-Vargas \cite{ChaJi} explored conditions under which Lipschitz maps admit extensions within certain classes of Lipschitz maps.

Our starting point combines the work of Ch\'avez-Dom\'{\i}nguez and Jim\'enez-Vargas with the theory of linear operator extensions. Consider a linear operator $T\colon E_0\rightarrow F$ and suppose that, for a Banach space $E\supset E_0$, no linear extension of $T$ exists from $E$ to $F$. Since $T$ is linear, in particular is Lipschitz. So, could it admit a Lipschitz extension from $E$ to $F$? A partial answer follows from a theorem due to Lindenstrauss \cite{Lin} and Pe\l czy\'nski \cite{Pel}, specifically \cite[Theorem~7.2]{BL}:

\begin{theo}
Let $E, F$ be Banach spaces and $E_0$ a subspace of $E$. Let $T\in \L(E_0,F)$, and suppose there exists a Lipschitz function $f\colon E\rightarrow F$ such that $f|_{E_0}=T$. Then, there exists an operator $R\in \L(E,F'')$ with $R|_{E_0}=T$ and $\|R\| \leq \Lip(f)$.
\end{theo}

Thus, if $E$ and $F$ are Banach spaces with $F$ a dual space and $E_0$ is a subspace of $E$, a linear operator $T\colon E_0\rightarrow F$ has a Lipschitz extension to $E$ if and only if it has a linear extension. Now, suppose a linear operator admits a Lipschitz extension with some property. Could the linear extension share an {\it analogous property}? Unfortunately, the answer is no. Returning to the canonical inclusion $J_{c_0}\colon c_0\rightarrow \ell_{\infty}$, since there exists a Lipschitz map $r\colon \ell_{\infty} \rightarrow c_0$ with $r|_{c_0}=Id_{c_0}$ \cite[Example~1.5]{BL} (i.e., $c_0$ is a Lipschitz retract of $\ell_{\infty}$), the composition $J_{c_0}\circ r \colon \ell_{\infty} \rightarrow \ell_{\infty}$ is a Lipschitz extension of $J_{c_0}$ that factorizes through $c_0$. However, as noted earlier, no linear extension of $J_{c_0}$ can factorize through $c_0$.

In this article, we focus on cases where, if a linear operator admits a Lipschitz extension with some {\it extra property}, it also has a linear extension with an analogous property. Before proceeding, we must clarify the terms {\it property} and {\it analogous property}. We work with Banach Lipschitz operator ideals $\I$ (defined below), with norm denoted by $\|\cdot\|_{\I}$. If we consider a Banach Lipschitz operator ideal $\I$, from the definitions it follows that $\I\cap \L$, the set of all linear operators belonging to $\I$, is a Banach linear operator ideal with the same norm.

The main theorem of this article is a generalization of \cite[Theorem~7.2]{BL}, where we consider linear and Lipschitz operators belonging to a specific ideal. Our main theorem is:

\begin{theoP}\label{Thm: Main}
Let $E, F$ be Banach spaces and let $E_0$ be a subspace of $E$. Let $T\in \L(E_0,F)$ and assume that there exists a Lipschitz function $f\colon E\rightarrow F$ such that $f|_{E_0}=T$. If $f$ belongs to a maximal Banach Lipschitz operator ideal $\I$ that is invariant under translations, then there exists an operator $R\in \I\cap \L(E,F'')$ such that $R|_{E_0}=T$ and $\|R\|_{\I} \leq \|f\|_{\I}$.
\end{theoP}

It is worth mentioning that, under the assumptions of the above theorem, the operator $T$ must belong to the ideal $\I \cap \L$. As a consequence of Theorem~\ref{Thm: Main}, we obtain:

\begin{theoP}\label{Thm: coro}
Let $E, F$ be Banach spaces and let $E_0$ be a subspace of $E$. Let $\A$ be a maximal Banach operator ideal and $T\in \A(E_0,F')$. The following are equivalent:
\begin{enumerate}
\item There exists $R\in \A(E,F')$ such that $R|_{E_0}=T$ and $\|R\|_{\A}=\|T\|_{\A}$.
\item For every maximal Banach Lipschitz operator ideal $\I$ invariant under translations such that $\I\cap \L=\A$, there exists a Lipschitz function $f\in \I(E,F')$ such that $f|_{E_0}=T$ and $\|f\|_{\I}=\|T\|_{\A}$.
\item\label{item} There exists a maximal Banach Lipschitz operator ideal $\I$ invariant under translations such that $\I\cap \L=\A$ and a Lipschitz function $f\in \I(E,F')$ with $f|_{E_0}=T$ and $\|f\|_{\I}=\|T\|_{\A}$.
\end{enumerate}
\end{theoP}

In summary, a linear operator with values in a dual space that belongs to a maximal Banach operator ideal has an extension in the same ideal if and only if it has a Lipschitz extension in some maximal Banach Lipschitz operator ideal related to the linear ideal. In Section~\ref{Sec: Main}, we provide all necessary definitions and preliminaries and prove Theorems \ref{Thm: Main} and \ref{Thm: coro}. In Theorem \ref{Thm: Main} and \ref{Thm: coro} there are some assumptions on the Lipschitz ideal: to be maximal and invariant under translations. In Section~\ref{Sec: Improve} we give a brief discussion about this two assumptions and discuss possible improvements of Theorems~\ref{Thm: Main} and \ref{Thm: coro}.

Before proceeding, let us mention a direct consequence of Theorem~\ref{Thm: Main}. Johnson, Maurey, and Schechtman obtained a linear factorization through a fixed Banach space $Z$ for a linear operator, assuming a Lipschitz factorization through the same space $Z$ and a differentiability condition on one of the Lipschitz maps involved \cite[Theorem~1]{JMS}. The following proposition shows that if a linear operator $T\colon E\rightarrow F'$ can be factorized via two Lipschitz maps through $\ell_{\infty}(B_{E'})$, no differentiability condition is needed.

\begin{propP}\label{Prop: Facto}
Let $E$ and $F$ be Banach spaces and $T\colon E\rightarrow F'$ be a linear operator. Suppose there exist two Lipschitz functions $f\colon E\rightarrow \ell_{\infty}(B_{E'})$ and $g\colon \ell_{\infty}(B_{E'})\rightarrow F'$, the latter belonging to a maximal Banach Lipschitz operator ideal $\I$ that is invariant under translations, such that $T=g\circ f$. Then, there exist linear operators $S\colon E\rightarrow \ell_{\infty}(B_{E'})$ and $R\colon \ell_{\infty}(B_{E'})\rightarrow F'$ in $\I\cap \L$ such that $T=R\circ S$ and $\|R\|_{\I} \|S\|\leq \Lip(f) \|g\|_\I$.
\end{propP}

\begin{proof}
Denote by $\iota_E \colon E \hookrightarrow \ell_{\infty}(B_{E'})$ the canonical inclusion. Take the linear operator $T$ as in the statement. Since $\ell_{\infty}(B_{E'})$ is a Lipschitz injective Banach space (see, e.g. \cite[Lemma 1.1]{BL}), there exists a Lipschitz map $\widetilde f\colon \ell_{\infty}(B_{E'}) \rightarrow \ell_{\infty}(B_{E'})$ with $\Lip(\widetilde f)=\Lip(f)$ such that $f=\widetilde f \circ \iota_E$. Then, $g\circ \widetilde f \in \I(\ell_{\infty}(B_{E'}),F')$ is an extension of $T$. Applying Theorem~\ref{Thm: Main}, there exists a linear operator $R\in \I\cap \L(\ell_{\infty}(B_{E'}),F')$ with $\|R\|_{\I}\leq \|g\circ \widetilde f\|_{\I}\leq \|g\|_{\I} \Lip(f)$ such that $R\circ \iota_E=T$, completing the proof.
\end{proof}

\section{Preliminaries and main theorems}\label{Sec: Main}

From now on, $X$ and $Y$ will denote pointed metric spaces and their distinguished points will always be denoted by $0$. We will use $E$ and $F$ to denote real Banach spaces. Every Banach space can be considered as a pointed metric space with distinguished point $0$. The open unit ball of $E$ will be denoted by $B_E$.

A map $f\colon E\rightarrow F$ is called a \textit{Lipschitz map} if there exists a constant $C>0$ such that $\|f(x_1)-f(x_2)\|\leq C \|x_1-x_2\|$ for all $x_1, x_2 \in E$. The smallest such constant is denoted by $\Lip(f)$. We denote by $\Lip_0(E,F)$ the set of all Lipschitz maps from $E$ to $F$ that vanish at $0$.

By a \textit{Banach Lipschitz operator ideal}, we mean a subclass $\I_{\Lip}$ of $\Lip_0$ such that for every pointed metric space $X$ and every Banach space $E$, the component
$$
\I_{\Lip}(X,E)=\Lip_0(X,E)\cap \I_{\Lip}
$$
satisfies:
\begin{enumerate}[\upshape (i)]
\item $\I_{\Lip}(X,E)$ is a linear subspace of $\Lip_0(X,E)$.
\item $Id_\mathbb{R} \in \I_{\Lip}(\mathbb{R},\mathbb{R})$.
\item Ideal property: if $g \in \Lip_0(Y,X)$, $f\in \I_{\Lip}(X,E)$ and $S \in \L(E,F)$, then the composition $S\circ f\circ g\in \I_{\Lip}(Y,F)$.
\end{enumerate}
Additionally, there is a \textit{Lipschitz ideal norm} on $\I_{\Lip}$, given by a function $\|{\cdot}\|_{\I_{\Lip}}\colon \I_{\Lip}\rightarrow [0,+\infty)$ that satisfies:

\begin{enumerate}[\upshape (i')]
\item For every pointed metric space $X$ and every Banach space $E$, the pair $\big(\I_{\Lip}(X,E);\|{\cdot}\|_{\I_{\Lip}}\big)$ is a Banach space and $\Lip(f)\leq \|f\|_{\I_{\Lip}}$ for all $f \in \I_{\Lip}(X,E)$.
\item $\|{Id_{\mathbb{R}} \colon \mathbb{R}\rightarrow \mathbb{R}}\|_{\I_{\Lip}}=1$.
\item If $g \in \Lip_0(Y,X)$, $f\in \I_{\Lip}(X,E)$ and $S \in \L(E,F)$, then $\|S\circ f\circ g\|_{\I_{\Lip}} \leq \Lip(g) \|f\|_{\I_{\Lip}} \|S\|$.
\end{enumerate}

This definition was introduced in \cite[Definition~2.1]{ARSPY} and independently in \cite[Definition~2.3]{CPCDJVVV}, under the name of \textit{generic Lipschitz operator Banach ideal}, and extends the definition of Banach linear operator ideals. As we already mention, for a Banach Lipschitz operator ideal $\I$, the class of linear operators that belong to $\I$ will be denoted by $\I \cap \L$. It is straightforward to see that $\I \cap \L$ is a Banach operator ideal endowed with the norm $\|\cdot\|_\I$. For the basics of Banach linear operator ideals, we refer to Pietsch's book \cite{Pie}.

Throughout the manuscript, a Banach operator ideal will refer to a Banach \textit{linear} operator ideal. To avoid confusion, we use $\I$ for Banach Lipschitz operator ideals and $\A$ for Banach operator ideals. The notation $\A\subset \B$ means that for all Banach spaces $E$ and $F$, $\A(E,F)\subset \B(E,F)$ and $\|\cdot\|_{\B}\leq \|\cdot\|_{\A}$. The same applies to Lipschitz operator ideals.

As in the linear case, for a Banach Lipschitz operator ideal $\I$, there exists the \textit{maximal hull} of $\I$, $\I^{\max}$. Following \cite[Lemma~3.3]{CPCDJVVV2}, a Lipschitz function $f\colon X\rightarrow E$ belongs to $\I^{\max}(X,E)$ if 
$$
\|f\|_{\I^{\max}} \colon= \sup \|Q \circ f \circ \iota\|_{\I(X_0,E/M)} < \infty,
$$
where the supremum is taken over all finite metric subspaces $X_0$ of $X$, all cofinite-dimensional subspaces $M$ of $E$, $\iota\colon X_0\rightarrow X$ is the inclusion, and $Q$ is the quotient map from $E$ to $E/M$.

There exists a theory of Lipschitz tensor norms, developed in \cite{CPCDJVVV}. Let $X$ be a pointed metric space and $E$ a Banach space. For $x, y \in X$ and $z\in E$, define $\delta_{x,y} \boxtimes z\colon \Lip_0(X,E')\rightarrow \mathbb{R}$ by
$$
(\delta_{x,y} \boxtimes z)(f) := (f(x)-f(y))(z)
$$ 
and 
$$
X\boxtimes E := \operatorname{Span} \{\delta_{x,y} \boxtimes z \colon x,y\in X, z \in E\} \subset \Lip_0(X,E')'.$$
 
For a \textit{finitely generated Lipschitz cross-norm} $\alpha$ (see \cite{CPCDJVVV} for the definition), we denote by $X\widehat{\boxtimes}_{\alpha} E$ the completion of $X \boxtimes E$ with respect to the norm $\alpha$. 

There is also a \textit{Representation Theorem for Lipschitz Maximal Operator Ideals} \cite[Corollary~5.2]{CPCDJVVV2}. For a maximal Banach Lipschitz operator ideal $\I$ there exists a finitely generated Lipschitz cross-norm $\alpha$ such that for every pointed metric space $X$ and every Banach space $E$,
\begin{equation}\label{eq1}
\I(X,E') = (X\widehat{\boxtimes}_{\alpha} E)'.
\end{equation}
In this case, we say that $\I$ and $\alpha$ are associated by duality. For $f \in \I(X,E')$ and $u = \sum_{i=1}^{n} \delta_{x_i,y_i} \boxtimes z_i \in X \boxtimes E$, the representation in \eqref{eq1} is given by
$$
\langle f, u \rangle = \sum_{i=1}^{n} (f(x_i) - f(y_i))(z_i).
$$

For a Lipschitz map $f \in \Lip_0(E,F)$ and $x_0 \in E$, we denote by $f^{x_0}(\cdot) \colon= f(\cdot + x_0) - f(x_0)$ the translation of $f$. If $X \subset E$ is a subset, we write $X - x_0 = \{x - x_0 \colon x \in X\}$. We need the following definition:

\begin{definition}
A Banach Lipschitz operator ideal $\I$ is \textit{invariant under translations} if for all Banach spaces $E$ and $F$, every subset $X \subset E$ with $0 \in X$, every $x_0 \in X$, and every Lipschitz map $f \colon X \rightarrow F$, we have $f \in \I(X,F)$ if and only if $f^{x_0} \in \I(X - x_0,F)$, with $\|f\|_{\I} = \|f^{x_0}\|_{\I}$.
\end{definition}

Note that if $\I$ is invariant under translations, then for every Banach spaces $E$ and $F$ and every $x_0 \in E$, the map $\phi \colon \I(E,F) \rightarrow \I(E,F)$ given by $\phi(f) = f^{x_0}$ is an isometric isomorphism, since $f = (f^{x_0})^{-x_0}$.

We do not have an example of a Banach Lipschitz operator ideal that is not invariant under translations. We will discuss this further in the next section.

\begin{lemma}\label{Lemma: IUT}
Let $\I$ be a Banach Lipschitz operator ideal. If $\I$ is invariant under translations, then $\I^{\max}$ is also invariant under translations.
\end{lemma}

\begin{proof}
Since $f = (f^{x_0})^{-x_0}$, it is enough to show that for Banach spaces $E$ and $F$, $X$ a subset of $E$ with $0 \in X$, and $x_0 \in X$, if $f \in \I^{\max}(X,F)$, then $f^{x_0} \in \I^{\max}(X - x_0,F)$ with $\|f^{x_0}\|_{\I^{\max}} \leq \|f\|_{\I^{\max}}$. 

Let $M$ be a cofinite-dimensional subspace of $F$ and $Y \subset X - x_0$ a finite metric space. Let $\iota_Y^{X - x_0}$ denote the inclusion of $Y$ into $X - x_0$, and $Q$ the quotient map from $F$ to $F/M$. Let $\widetilde{Y} = Y \cup \{-x_0\}$ and note that $\widetilde{Y} \subset X - x_0$. Since $\I$ is invariant under translations, we have
\[
\begin{aligned}
\left\|Q \circ f^{x_0} \circ \iota_Y^{X - x_0}\right\|_{\I(Y,F/M)} &\leq \left\|Q \circ f^{x_0} \circ \iota_{\widetilde{Y}}^{X - x_0}\right\|_{\I(\widetilde{Y},F/M)} \\
&= \left\|\left(Q \circ f \circ \iota_{\widetilde{Y} + x_0}^X\right)^{x_0}\right\|_{\I(\widetilde{Y},F/M)} \\
&= \left\|Q \circ f \circ \iota_{\widetilde{Y} + x_0}^X\right\|_{\I(\widetilde{Y} + x_0, F/M)} \\
&\leq \|f\|_{\I^{\max}(X,F)}.
\end{aligned}
\]
The last inequality holds because $M$ is cofinite-dimensional in $F$ and $\widetilde{Y} + x_0$ is a finite subset of $X$. Thus, $\|f^{x_0}\|_{\I^{\max}(X - x_0,F)} \leq \|f\|_{\I^{\max}(X,F)}$. 
\end{proof}

For the proof of Theorem~\ref{Thm: Main}, we will need the following lemma:

\begin{lemma}\label{LemmaTranslationInvariance}
Let $\I$ be a maximal Banach Lipschitz operator ideal that is invariant under translations, and let $\alpha$ be the finitely generated tensor norm associated to $\I$ by duality. For any Banach spaces $E$ and $F$, elements $x_1, y_1, \dots, x_n, y_n \in E$, and $z_1, \dots, z_n \in F$, we have
\[
\alpha\left(\sum_{i=1}^n \delta_{x_i + x_0, y_i + x_0} \boxtimes z_i\right) = \alpha\left(\sum_{i=1}^n \delta_{x_i, y_i} \boxtimes z_i\right)
\]
for all $x_0 \in E$.
\end{lemma}

\begin{proof}
Fix $x_0 \in E$. By the Representation Theorem for Maximal Banach Lipschitz Operator Ideals and since $\I$ is invariant under translations, we have
\[
\begin{aligned}
\alpha\left(\sum_{i=1}^n \delta_{x_i + x_0, y_i + x_0} \boxtimes z_i\right) &= \sup_{f \in B_{\I(E,F')}} \left|\left\langle f, \sum_{i=1}^n \delta_{x_i + x_0, y_i + x_0} \boxtimes z_i \right\rangle\right| \\
&= \sup_{f \in B_{\I(E,F')}} \left|\sum_{i=1}^n (f(x_i + x_0) - f(y_i + x_0))(z_i)\right| \\
&= \sup_{f \in B_{\I(E,F')}} \left|\sum_{i=1}^n (f^{x_0}(x_i) - f^{x_0}(y_i))(z_i)\right| \\
&= \sup_{f \in B_{\I(E,F')}} \left|\left\langle f^{x_0}, \sum_{i=1}^n \delta_{x_i, y_i} \boxtimes z_i \right\rangle\right| \\
&= \alpha\left(\sum_{i=1}^n \delta_{x_i, y_i} \boxtimes z_i\right).
\end{aligned}
\]
\end{proof}

Now we are ready to prove Theorem~\ref{Thm: Main}. The proof follows the same pattern as \cite[Theorem~7.2]{BL}.

\begin{proof}[Proof of Theorem~\ref{Thm: Main}]
First, we assume that $E$ is a finite dimensional Banach space and write (algebraically) $E=E_{0}\oplus E_{1}$. Let $\varphi\geq 0$ be a $C^{\infty}$ function with compact support on $E_{0}$ such that $\int_{E_{0}}\varphi=1$ and $\varphi(x)=\varphi(-x)$ for all $x\in E_{0}$. For $z\in E$, we define
\begin{align*}
g(z)=\int_{E_{0}}f(z+x)\varphi(x)dx.
\end{align*}
We have $g(0)=\int_{E_{0}}f(x)\varphi(x)dx=\int_{E_{0}}T(x)\varphi(x)dx=0$ because $T(x)\varphi(x)=-T(-x)\varphi(-x)$. Since $f$ is Lipschitz, $g$ is Lipschitz as well. 

To show that $g\in \I(E,F)$ with $\|g\|_{\I}\leq \|f\|_{\I}$, note that since $\I$ is maximal, by \cite[Corollary~5.3]{CPCDJVVV2} it is regular. Thus it suffices to show that $k_{F}\circ g\in \I(E,F'')$ with $\|k_{F}\circ g\|_{\I}\leq \|k_F\circ f\|_{\I}$, where $k_F\colon F\rightarrow F''$ is the canonical inclusion. 

By linearity of $k_{F}$, we have for $z\in E$
\begin{align*}
(k_{F}\circ g)(z)=\int_{E_{0}}(k_{F}\circ f)(z+x)\varphi(x)dx.
\end{align*}

By \cite[Theorem~5.1]{CPCDJVVV2}, there exists a finitely generated Lipschitz tensor norm $\alpha$ such that $\I(E,F'')=(E\widehat{\boxtimes}_{\alpha} F')'$. Take $u=\sum_{i=1}^{n}\delta_{x_{i},y_{i}}\boxtimes v_{i}'$ in $E\boxtimes F'$, where $x_{i},y_{i}\in E$ and $v_{i}'\in F'$ for $i=1,\dots,n$. We have
\begin{align*}
|\left\langle k_{F}\circ g,u\right\rangle|&=\left|\sum_{i=1}^{n}\Big\langle (k_{F}\circ g)(x_{i})-(k_{F}\circ g)(y_{i}), v_{i}'\Big\rangle\right|
\\
&=\left|\sum_{i=1}^{n}\left\langle \int_{E_{0}}\Big((k_{F}\circ f)(x_{i}+x)-(k_{F}\circ f)(y_{i}+x)\Big)\varphi(x) dx, v_{i}'\right\rangle\right|
\\
&=\left|\int_{E_{0}}\left(\sum_{i=1}^{n}\left\langle \Big((k_{F}\circ f)(x_{i}+x)-(k_{F}\circ f)(y_{i}+x)\Big), v_{i}'\right\rangle\right) \varphi(x)dx\right|
\\
&\leq \int_{E_{0}}\left|\left\langle k_{F}\circ f, \sum_{i=1}^{n}\delta_{x_{i}+x, y_{i}+x}\boxtimes v_{i}'\right\rangle\right|\varphi(x)dx
\\
&\leq \|k_F\circ f\|_{\I}\int_{E_{0}}\alpha\left(\sum_{i=1}^{n}\delta_{x_{i}+x, y_{i}+x}\boxtimes v_{i}'\right)\varphi(x)dx
\\
&=\|k_F\circ f\|_{\I} \alpha(u),
\end{align*}
where the last equality follows from Lemma~\ref{LemmaTranslationInvariance} and the fact that $\int_{E_0}\varphi(x) dx=1$. Thus $g\in \I(E,F)$ with $\|g\|_{\I}\leq \|f\|_{\I}$.

Now, given $z\in E$ and a direction $y\in E_{0}$, setting $u=ty+x$ where $x\in E_{0}$, we obtain
\begin{align*}
g(z+ty)=\int_{E_{0}}f(z+ty+x)\varphi(x)dx=\int_{E_{0}}f(z+u)\varphi(u-ty)du
\end{align*}
and consequently
\begin{align*}
\frac{g(z+ty)-g(z)}{t}=\int_{E_{0}}f(z+u)\left(\frac{\varphi(u+t(-y))-\varphi(u)}{t}\right)du.
\end{align*}

Since $\varphi$ is a $C^{\infty}$ function with compact support on $E_{0}$, by the Dominated Convergence Theorem for the Bochner integral (see \cite[Theorem II.2.3]{DieUl}), the directional derivative of $g$ at $z$ in the direction $y$ exists, will be denoted by $\frac{\partial g}{\partial y}(z)$, and satisfies
\begin{align}
\label{eq:directional_derivative}
\frac{\partial g}{\partial y}(z)=-\int_{E_{0}}f(z+x)D\varphi(x)(y)\; dx,
\end{align}
where $z\in E$ and $y\in E_{0}$. Note that this expression is continuous in $z\in E$ and linear in $y\in E_{0}$.

Additionally, for every $y\in E_{0}$, we have
\begin{align*}
g(y) &= \int_{E_{0}}f(y+x)\varphi(x)dx \\
&= \int_{E_{0}}T(y+x)\varphi(x)dx 
\\
&= T(y)\int_{E_{0}}\varphi(x)dx + \int_{E_{0}}T(x)\varphi(x)dx \\
&= T(y).
\end{align*}
In summary, we obtain $g\in \I(E,F)$ with $\|g\|_{\I}\leq \|f\|_{\I}$, $g|_{E_0}=T$, and it satisfies \eqref{eq:directional_derivative}.

Now, let $\psi\geq 0$ be a $C^{\infty}$ function with compact support on $E_{1}$ such that $\int_{E_{1}}\psi=1$, and define for $z\in E$
\begin{align*}
f_{n}(z) = \int_{E_{1}}g\left(z+\frac{u}{n}\right)\psi(u)du
\end{align*}
and for $n \in \mathbb N$
\begin{align*}
h_n(z) \colon= f_{n}(z)-f_{n}(0) &= \int_{E_{1}}\left(g\left(z+\frac{u}{n}\right)-g\left(\frac{u}{n}\right)\right)\psi(u)du \\
&= \int_{E_{1}}g^{\frac{u}{n}}(z)\psi(u)du.
\end{align*}
Since $\I$ is invariant under translations, $g^{\frac{u}{n}}\in \I(E,F)$ with $\|g^{\frac{u}{n}}\|_{\I}\leq \|g\|_{\I}$ for all $u\in E_{1}$ and $n\in \mathbb{N}$. As before, we obtain $h_n\in \I(E,F)$ with $\|h_n\|_{\I} \leq \|g\|_{\I}$. 

We now show that $h_n$ is Gateaux differentiable at $z=0$, for every $n \in \mathbb N$. Let $y=y_0+y_1$ where $y_{0}\in E_{0}$ and $y_{1}\in E_{1}$. Setting $w=t y_{1}+\frac{u}{n}$ for $u\in E_{1}$, we have
\begin{align*}
f_{n}(ty) &= \int_{E_{1}}g(t y_{1}+ty_{0}+\frac{u}{n})\psi(u)du \\
&=n \int_{E_{1}}g(w+t y_{0})\psi(nw-t n y_{1})  dw
\end{align*}
and setting $w=\frac{u}{n}$
\begin{align*}
f_{n}(0) &= \int_{E_{1}}g\left(\frac{u}{n}\right)\psi(u)du = n\int_{E_{1}}g(w)\psi(n w)dw.
\end{align*}
Therefore,
\begin{align*}
\frac{h_n(ty)}{t} &= \frac{f_{n}(ty)-f_{n}(0)}{t} \\
&= \frac{n}{t}\int_{E_{1}}(g(u+t y_{0})\psi(n u-t n y_{1})-g(u)\psi(nu))du \\
&= n\int_{E_{1}}\psi(n u-t n y_{1})\left(\frac{g(u+t y_{0})-g(u)}{t}\right)du 
+ n\int_{E_{1}}g(u)\left(\frac{\psi(n u-t n y_{1})-\psi(nu)}{t}\right)du.
\end{align*}
Again, by the Dominated Convergence Theorem, we obtain
\begin{equation}\label{Gateaux}
\lim_{t\rightarrow 0}\frac{h_n(ty)}{t} = n\int_{E_{1}}\psi(nu)\frac{\partial g}{\partial y_{0}}(u)du + n\int_{E_{1}}g(u)D\psi(nu)(-n y_{1})du.
\end{equation}
The right-hand side expression is linear in $(y_{0},y_{1})\in E_{0}\oplus E_{1}=E$, and since $E$ is a finite dimensional space, it defines a bounded linear operator from $E$ to $F$. Consequently, $h_{n}$ is Gateaux differentiable at $z=0$ fro every $n\in \mathbb N$. Also, since $\I$ is maximal, by \cite[Lemma 4.7]{AlTur}, we have that the Gateaux derivative of $h_n$ at $0$, $Dh_n(0)$, belongs to $ \I\cap \L(E,F)$ with $\|Dh_n(0)\|_{\I}\leq \|h_{n}(\cdot)\|_{\I}\leq \|g\|_{\I}$.

Consider the operators $k_F\circ Dh_n(0)\in \|g\|_{\I} B_{\I\cap \L(E,F'')}$. By \cite[Lemma 4.1]{AlTur}, $\I\cap \L(E, F'')$ is a maximal Banach operator ideal, and by the Representation Theorem for Maximal Operator Ideals, it is a dual space. Thus, by passing to a subnet if necessary, there exists $R \in \I\cap \L(E,F'')$ with $\|R\|_{\I\cap \L}\leq \|g\|_{\I}$ such that $R=w^{*}\text{-}\lim_{n\to \infty} k_F\circ Dh_n(0)$. Equivalently, for every $x\in E$ and $y'\in F'$, $R(x)(y')=\lim_{n\to \infty} \Big(k_F\circ Dh_n(0)(x)\Big)(y')$. 

We now show that $R|_{E_{0}}=T$. For $y_0\in E_{0}$, first by \eqref{Gateaux} we have
\begin{align*}
Dh_{n}(0)(y_0) =\int_{E_{1}}\frac{\partial g}{\partial y_0}\left(\frac{u}{n}\right)\psi(u)du.
\end{align*}
On the other hand,  since $g|_{E_0}=T$, we have $\dfrac{\partial g}{\partial y_0}(0)=Tx$, and therefore for $y' \in F'$:
\begin{align*}
R(y_0)(y')-y'(T(y_0)) &= \lim_{n\to \infty}(Dh_{n}(0)(y_0)-Ty_0)(y') \\
&= \lim_{n\to \infty}\int_{E_{1}}\left(\frac{\partial g}{\partial y_0}\left(\frac{u}{n}\right)-\frac{\partial g}{\partial y_0}(0)\right)(y')\psi(u) du.
\end{align*}
By the continuity of $\dfrac{\partial g}{\partial y_0}(\cdot)$ and the compact support of $\psi$ we have 
\begin{align*}
\lim_{n\to \infty}\int_{E_{1}}\left(\frac{\partial g}{\partial y_0}\left(\frac{u}{n}\right)-\frac{\partial g}{\partial y_0}(0)\right)(y')\psi(u) du &= 0.
\end{align*}
Since $y' \in F'$ was arbitrary, we conclude that $k_F\circ Ty_0=Ry_0$. This completes the proof for the finite dimensional case.

For the general case, denote by $\text{FIN}(E)$ the family of all finite-dimensional subspaces of $E$, ordered by inclusion. For each $G\in \text{FIN}(E)$, by the finite-dimensional case, there exists $R_{G}\in \I\cap \L(G,F'')$ such that $R_{G}|_{G\cap E_{0}}=T|_{G\cap E_{0}}$ and $\|R_{G}\|_{\I\cap \L}\leq \|f|_G\|_{\I}\leq \|f\|_{\I}$ (hence $\|R_{G}\|\leq \|f\|_{\I}$). Define $h_{G}\colon E\rightarrow F''$ by:
\begin{align*}
h_{G}x = \begin{cases}
R_{G}x & \text{if }x\in G \\
0 & \text{if }x\notin G.
\end{cases}
\end{align*}

For each $x\in E$, we have $\|h_{G}x\|\leq \|f\|_{\I}\|x\|$. Fix an ultrafilter $\mathcal{U}$ on $\text{FIN}(E)$ containing the order filter. By the $w^{*}$-compactness of the closed ball of $F''$, the limit
\begin{align*}
Rx \colon= w^{*}\text{-}\lim_{\mathcal{U}}h_{G}x
\end{align*}
exists for every $x\in E$. We verify that $R\colon E\rightarrow F''$ is linear and extends $T$. For $x,y\in E$ and all $G\in \text{FIN}(E)$ containing $\text{span}\{x,y\}$:
\begin{align*}
h_{G}(x+y) = R_{G}(x+y) = R_{G}x + R_{G}y = h_{G}x + h_{G}y,
\end{align*}
which implies $R(x+y)=Rx+Ry$. For $x\in E_{0}$ and all $G\in \text{FIN}(E)$ containing $x$:
\begin{align*}
h_{G}x = R_{G}x = k_F\circ Tx,
\end{align*}
hence $Rx=k_F\circ Tx$.

To conclude, we show $R\in \I\cap \L(E,F'')$ with $\|R\|_{\I\cap \L}\leq \|f\|_{\I}$. Since $\I\cap \L$ is a maximal Banach operator ideal, there exists a finitely generated tensor norm $\beta$ such that $\I\cap \L(E,F'')=(E\widehat \otimes_{\beta}F')'$. Fix $u\in E\otimes F'$ with representation $u=\sum_{i=1}^{n}x_{i}\otimes y_{i}'$. Define, for $G\in \text{FIN}(E)$
\begin{align*}
t_{G} = \begin{cases}
\beta(u,G\otimes F') & \text{if }\text{span}\{x_{i}\}_{i=1}^n\subseteq G \\
0 & \text{if } \text{span}\{x_{i}\}_{i=1}^n\nsubseteq G
\end{cases}
\end{align*}
where $\beta(u,G\otimes F')$ is the $\beta$-norm of $u$ in $G\otimes F'$. By the metric mapping property of $\alpha$ (see, e.g.\cite[12.1]{DF})
\begin{align*}
\beta(u, G\otimes F')\leq \beta(u, \text{span}\{x_i\}_{i=1}^n\otimes F').
\end{align*}
This implies that the net  $(t_{G})_G$ is bounded, hence $\lim_{\mathcal{U}}t_{G}$ exists. Since $\beta$ is finitely generated, for $\varepsilon>0$ there exists $G_{\varepsilon}\in \text{FIN}(E)$ with $\beta(u, G_{\varepsilon}\otimes F') \leq \beta(u,E\otimes F')+\varepsilon$. Let $\widetilde G_{\varepsilon}=\text{span}\{G_{\varepsilon} \cup \{x_i\}_{i=1}^n\}$. Then
\begin{align*}
\beta(u,E\otimes F') &\leq \beta(u,\widetilde G_{\varepsilon} \otimes F') \\
&\leq \beta(u,G_{\varepsilon}\otimes F') \\
&\leq \beta(u,E\otimes F')+\varepsilon.
\end{align*}
For all $G_{1}\supset \widetilde G_{\varepsilon}$ we obtain
\begin{align*}
\beta(u, E\otimes F') \leq t_{G_{1}} \leq \beta(u, E\otimes F')+\varepsilon,
\end{align*}
implying that
\begin{align*}
\beta(u, E\otimes F') \leq \lim_{\mathcal{U}}t_{G} \leq \beta(u,E\otimes F')+\varepsilon.
\end{align*}
Since $\varepsilon>0$ is arbitrary we conclude that $\lim_{\mathcal{U}}t_{G}=\beta(u, E\otimes F')$. Finally, for all finite dimensional subspace $G$ or $E$ containing $\text{span}\{x_{i}\}_{i=1}^n$
\begin{align*}
\left|\sum_{i=1}^{n}\langle h_{G}x_{i},y_{i}'\rangle\right| &= \left|\sum_{i=1}^{n}\langle R_{G}x_{i},y_{i}'\rangle\right| \\
&= |\langle R_{G}, u\rangle| \\
&\leq \|R_{G}\|_{\I\cap \L}\alpha(u, G\otimes F') \\
&\leq \|f\|_{\I}t_{G}.
\end{align*}
Taking limits over $\mathcal{U}$ on the both sides of the inequality. we obtain
\begin{align*}
|\langle R,u\rangle| \leq \|f\|_{\I}\beta(u, E\otimes F').
\end{align*}
This shows $R\in \I\cap \L(E,F'')$ with $\|R\|_{\I\cap \L}\leq \|f\|_{\I}$, completing the proof.
\end{proof}
A consequence of Theorem~\ref{Thm: Main} is Theorem~\ref{Thm: coro}.
\begin{proof}[Proof of Theorem~\ref{Thm: coro}]
The implication (1) $\Rightarrow$ (2) is immediate, since any linear extension of $T$ in $\A$ automatically belongs to every Banach Lipschitz operator ideal $\I$ satisfying $\I \cap \L = \A$. 

The implication (2) $\Rightarrow$ (3) follows from the existence of suitable Lipschitz ideals. Specifically, the Lipschitz ideal $\I = \A \circ \Lip_{0}$ is maximal by \cite[Theorem~4.1]{TV}, invariant under translations by \cite[Example~4.4]{AlTur} and satisfies $\I \cap \L = \A$ by \cite[Theorem~3.3]{TV}.

Finally, the implication (3) $\Rightarrow$ (1) is a direct consequence of Theorem~\ref{Thm: Main}, which provides the required linear extension.
\end{proof}

\section{Final remarks}\label{Sec: Improve}

Regarding Theorem~\ref{Thm: coro}, we observe that statement (2) requires two key properties of the Banach Lipschitz ideal: maximality and invariance under translations.

Let $E$ and $F$ be Banach spaces, $x_0 \in E$, $\I$ a Banach Lipschitz operator ideal, and $f \in \I(E, F)$. Denote by $\tau_{x_0}$ the translation by $x_0$, that is, the (affine) function $\tau_{x_0} \colon E \to E$ defined by $\tau_{x_0}(x) = x + x_0$. That $\I$ is invariant under translations means that for every $f \in \I(E, F)$, the composition $\tau_{-f(x_0)} \circ f \circ \tau_{x_0}$ also belongs to $\I(E, F)$.  Note that $\tau_{x_0} \notin \L(E, E)$ since it does not preserve the origin. Consequently, the ideal property by itself does not guarantee that $\tau_{-f(x_0)} \circ f \circ \tau_{x_0} \in \I(E, F)$. However, as far as we know, every Banach Lipschitz operator ideals satisfy this property.

Note that in the proof of Theorem~\ref{Thm: coro}, the implications (1) $\Rightarrow$ (2) and (2) $\Rightarrow$ (3) do not require that $\I$ being maximal. Thus, for Banach spaces $E$ and $F$, a subspace of $E$, $E_0$, and $T \in \L(E_0, F)$, if there exists $f \in \I(E, F)$ extending $T$, then $T$ admits an extension in $\I^{\max}$ (since $\I \subset \I^{\max}$) and by the implication (3) $\Rightarrow$ (1) of Theorem~\ref{Thm: coro}, $T$ has a linear extension in $\I^{\max} \cap \L$. However, despite of $\I \cap \L$ is maximal, we cannot conclude that $\I^{\max} \cap \L = \I \cap \L$. This leads to:

\begin{proposition}\label{Thm: coro2}
Let $E,F$ be Banach spaces, $E_0 \subset E$ a subspace, and $\A$ a maximal Banach operator ideal. For $T \in \A(E_0, F')$ and $\I$ a Banach Lipschitz operator ideal which is invariant under translations  with $\I^{\max} \cap \L = \I \cap \L = \A$, the following are equivalent:
\begin{enumerate}
    \item There exists $R \in \A(E, F')$ extending $T$ with $\|R\|_{\A} = \|T\|_{\A}$.
    \item There exists $f \in \I(E, F')$ extending $T$ with $\|f\|_{\I} = \|T\|_{\A}$.
\end{enumerate}
\end{proposition}

\begin{proof}
That (1) $\Rightarrow$ (2) follows identically to Theorem~\ref{Thm: coro}.

Now suppose that (2) holds, then the function $f$ is an extension that, in particular, belongs to $\I^{\max}$, which is invariant under translations by Lemma~\ref{Lemma: IUT}. Since $\I^{\max} \cap \L = \A$ by hypothesis, an application of Theorem~\ref{Thm: coro} completes the proof.
\end{proof}

As far as we know, all Banach Lipschitz operator ideals satisfy the hypotheses of Proposition~\ref{Thm: coro2}.

\subsection*{Acknowledgments} N. Albarrac\'{\i}n and P. Turco were supported in part by CONICET PIP 112202001101609CO. N. Albarrac\'{\i}n is also supported by a CONICET doctoral fellowship. P. Turco is also supported UBACyT 20020220300242BA.

\end{document}